\documentclass[11pt,a4paper]{amsart}

\usepackage[
  bookmarks=true,
  colorlinks=true,
  linkcolor=blue,
  urlcolor=blue,
  citecolor=blue,
  hyperindex=true
]{hyperref}
\usepackage{amsmath}
\numberwithin{equation}{section}
\usepackage{amsthm}
\usepackage{amssymb}
\usepackage{cite}
\usepackage{bm}
\usepackage{enumitem}
\usepackage{geometry}
\usepackage{indentfirst}
\usepackage{graphicx}
\usepackage{xcolor}
\usepackage{marginnote}
\newtheorem{theorem}{Theorem}[section]
\newtheorem{proposition}[theorem]{Proposition}
\newtheorem{example}[theorem]{Example}
\newtheorem{corollary}[theorem]{Corollary}
\newtheorem{lemma}[theorem]{Lemma}

\newtheorem{remark}[theorem]{Remark}
\newtheorem{definition}[theorem]{Definition}

\newcommand{\R}{\mathbb{R}}
\newcommand{\Z}{\mathbb{Z}}

\newcommand{\one}{\mathbf{1}}

\newcommand{\cites}[1]{\cite{#1}}

\title[]{Primitive Periods in Greedy Chip-Firing Games}
\author{Zheng Huang}
\address{School of Mathematical Sciences, CMA-Shanghai, Shanghai Jiao Tong University}
\email{hzxs280457@sjtu.edu.cn}
\author{Anyuan Tian}
\address{Laboratory of Combinatorial and Geometric Structures, Moscow Institute of Physics and Technology}
\email{tyan.a@mipt.ru}

\begin{document}

\begin{abstract}
For a finite irreducible rational transition matrix $P=(P_{ij})$ on $n$ states, the hunger game introduced by Li and Propp is a greedy chip-firing process in which each state $i$ carries a real-valued hunger $h_i$. At each step, a state $i$ of maximal hunger is selected, with ties resolved by a fixed order on the states. Firing $i$ subtracts one from its hunger and then adds $P_{ij}$ to the hunger of each state $j$. Let $\pi$ be the stationary distribution of $P$, and let $T$ be the least positive integer such that $T\pi$ is integral.
We prove that every periodic orbit has least period $T$, with state $i$ firing $T\pi_i$ times, and that the set of periodic states of total
hunger zero tiles the corresponding hyperplane by lattice translations. These results settle two conjectures of Li and Propp. We also prove an analogous least-period theorem for nontrivial periodic orbits of fixed-priority chip-firing on finite strongly connected digraphs, another greedy chip-firing model. We then find a connection between the hunger game and the
chairman assignment problem: the greedy rule that always chooses
the state furthest below its proportional target is precisely
a rank-one hunger game. Motivated by this application and related allocation rules,
we generalize the hunger game by assigning each state a score
that is a nondecreasing function of its hunger and firing a state
of maximal score each time.
For these monotone-score hunger games, we prove that every
periodic orbit still has least period $T$ and give a sufficient
condition for eventual periodicity from every initial state.
\end{abstract}

\maketitle

\section{Introduction}
\label{sec:introduction}

A finite Markov chain is specified by local transition probabilities, whereas its long-term behavior is governed by global objects such as stationary distributions, occupation measures, and hitting probabilities. Random trajectories provide one way to access these objects; for example, Markov chain Monte Carlo uses them for sampling, approximate counting, and integration \cites{JerrumSinclair1996,Diaconis2009}. A complementary combinatorial question asks what remains of the global behavior when the random choices are replaced by reproducible local updates.

Several deterministic analogues of random processes have been developed from
this viewpoint. Engel's probabilistic abacus converts rational transition data into deterministic integral redistributions \cites{Engel1975,Engel1976}. Rotor-routing replaces each random choice of an outgoing edge by a fixed cyclic schedule \cite{HolroydPropp}; functional routers extend this mechanism to general
finite-state Markov chains, including chains with irrational transition
probabilities \cite{ShiragaEtAl2018}. These models also belong to the combinatorial theory of discrete dynamical systems. They are closely connected with chip-firing, in which a firing redistributes chips along the outgoing edges of a graph, and with the more general framework of abelian networks \cites{HolroydLevineMeszarosPeresProppWilson,CorryPerkinson, BondLevine2016}. 
Chip-firing has connections with many areas of mathematics,
including algebraic graph theory, spanning trees, statistical physics,
potential theory on graphs, and graph-theoretic divisor theory; see
\cite{CorryPerkinson} for a systematic introduction.

Li and Propp~\cite{LiPropp} introduced the \emph{hunger game} as a ``greedy'' chip-firing model associated with a Markov chain. Let $P$ be a row-stochastic matrix on $V=\{1,\ldots,n\}$, fix a total order $\prec$ on $V$, and put $H=P-I$. For $i\in V$, write $P_i$ and $H_i$ for the $i$th rows of $P$ and $H$, respectively,  and let $e_i\in\R^V$ be the row vector whose $i$th coordinate is $1$ and whose other coordinates are $0$. A hunger state is a row vector $h=(h_1,\ldots,h_n)\in\R^n$. If $i$ is the $\prec$-least state whose coordinate $h_i$ is maximal, then the hunger map replaces $h$ by $h+P_i-e_i=h+H_i$. Following Li and Propp's chip interpretation of the dynamics \cite[Section~2]{LiPropp}, we call this update a \emph{firing to} $i$. Denote the resulting map by $F$. The total order resolves ties and makes the dynamics deterministic. Since every row of $H$ has coordinate sum zero, total hunger is conserved; after translating by a multiple of the all-ones vector, one may work in $Z\doteq\{h\in\R^n:\sum_i h_i=0\}$. For an initial hunger state $h$, its \emph{forward orbit} is the
sequence $(F^t(h))_{t\ge0}$.  Let $i_t(h)$ be the state fired at
$F^t(h)$.  The \emph{firing vector} of the first $t$ steps is the row
vector $v_h(t)\in\Z_{\ge0}^n$ whose $j$th coordinate is the number of
indices $s\in\{0,\ldots,t-1\}$ for which $i_s(h)=j$.  A state $h$ is
\emph{recurrent} if $F^N(h)=h$ for some positive integer $N$.  The
least such $N$ is the \emph{least period} of $h$, and
$\{h,F(h),\ldots,F^{N-1}(h)\}$ is its \emph{recurrent orbit}.  The
forward orbit of $h$ is \emph{eventually periodic} if it eventually
enters a recurrent orbit, meaning that $F^{t_0}(h)$ is recurrent for
some integer $t_0\ge0$.

Assume that $P$ is irreducible, and let $\pi$ be its stationary
distribution.  Li and Propp proved that after $t$ steps the proportion
of firings to each state differs from the corresponding coordinate of
$\pi$ by $O(t^{-1})$ \cite[Theorem~5.1]{LiPropp}.  They also proved
that when $P$ is rational, every forward orbit is bounded and
eventually periodic \cite[Lemma~3.2 and Remark~3.3]{LiPropp}. Suppose $P$ is rational; let
$T$ be the least positive integer such that $T\pi$ is integral,  and set $p=T\pi$.
They showed that the all-zero hunger state has least period $T$ and
conjectured that every recurrent state has the same least period
\cite[Proposition~7.3 and Conjecture~7.4]{LiPropp}.

Our first main result proves the conjecture of Li and Propp by showing that every
recurrent state has least period $T$, and its firing vector over one period is $p$ (Theorem~\ref{thm:conj74}).  If $h$ is recurrent with least period $N$, then the return identity $v_h(N)H=0$ and the irreducibility of $P$ immediately imply that
$v_h(N)=qp$ and $N=qT$ for some positive integer $q$; see also
\cite[Remark~3.3]{LiPropp}.  Thus $T$ is the smallest period
compatible with stationarity and integrality, and our theorem shows
that the greedy firing rule always attains this theoretical minimum.
The proof reduces to excluding $q>1$.  For a hypothetical orbit of
period $qT$, we introduce a \emph{quota-return map} on the finite set
of orbit points at which the maximum hunger attains its minimum. With this map in
place, the nonprimitive case $q>1$ is ruled out by a short argument.

The companion result concerns the geometry of the recurrent states.
Let $B$ be the set of recurrent states of the hunger map in $Z$;
following Li and Propp, we call $B$ the \emph{basin of attraction}
\cite[Section~7]{LiPropp}.  Let
$L=\{zH:z\in\Z^n,\ \sum_i z_i=0\}$, the lattice generated by the
differences $H_i-H_j$ of the firing increments.  Li and Propp proved
that $B\cap(h+L)\ne\varnothing$ for every $h\in Z$
\cite[Proposition~7.7]{LiPropp}, and conjectured that this intersection
always consists of exactly one point
\cite[Conjecture~7.5]{LiPropp}.  We prove this uniqueness statement in
Theorem~\ref{thm:conj75}, using a two-orbit analogue of the
quota-return map used to prove Theorem~\ref{thm:conj74}.  Equivalently, $
 Z=\bigsqcup_{u\in L}(B+u)$.
Thus $B$ is a bounded tile for the lattice $L$.  Moreover, if $h\in Z$
and $b$ is the unique recurrent state in $B\cap(h+L)$, then
$F^{kT}(h)=b$ for every sufficiently large integer $k$.

The same primitive-return phenomenon occurs in a second deterministic chip-firing model. In classical sequential chip-firing, any currently legal vertex may fire; the schedule is therefore part of the dynamics \cites{BjornerLovaszShor,BjornerLovaszDirected}. We fix a total order on the vertices of a finite strongly connected digraph and always fire the first legal vertex in this order. If $\Delta_D$ is the firing matrix and $\rho$ is the primitive positive integer vector in its left kernel, then every legal return word has firing vector $q\rho$ for some positive integer $q$. We call a periodic orbit \emph{active} if the rule fires at every point on it. Theorem~\ref{thm:priority-period} shows that every active periodic orbit has $q=1$, and hence least period $\sum_v\rho_v$, the smallest value allowed by the integral kernel of $\Delta_D$.

The hunger game also appears naturally in proportional allocation. Given positive target proportions $\alpha=(\alpha_1,\ldots,\alpha_d)$ with $\sum_{i=1}^{d}\alpha_{i}=1$, the chairman assignment problem asks for a sequence whose running allocation counts stay close to $t\alpha$ \cites{Niederreiter1972,Tijdeman1980}. After $t$ assignments, the deficit of state $i$ is its target
$t\alpha_i$ minus the number of times it has been chosen. We find
 that the rule that always chooses a state of largest current deficit,
with ties resolved by the fixed order, is precisely the hunger game
for the rank-one matrix with every row equal to $\alpha$. Other chairman-assignment algorithms use the same deficit update with a different selector. A classical example is Tijdeman's
algorithm \cite{Tijdeman1980}, which takes both current deficits
and target proportions into account when choosing the next state.

Motivated by these alternative selection rules, we generalize the hunger game by replacing each hunger coordinate $h_i$ by a nondecreasing state-dependent score $\varphi_i(h_i)$ when choosing the next firing. We prove that every recurrent orbit of the resulting monotone-score hunger game still has the primitive period $T$ (Theorem~\ref{thm:monotone-score-period}). Under a pairwise non-domination condition, every orbit is bounded and, for rational $P$, eventually periodic (Theorem~\ref{thm:monotone-score-eventual}).  This generalized hunger-game framework includes the chairman-assignment rules discussed above and may also help us understand greedy algorithms arising in other practical problems.

We begin in Section~\ref{sec:setup} by introducing the necessary
notation and proving Theorem~\ref{thm:conj74}.
In Section~\ref{sec:recurrent}, we study the recurrent set and
prove Theorem~\ref{thm:conj75}.
We then consider fixed-priority chip-firing in
Section~\ref{sec:priority-chip-firing}, where we prove
Theorem~\ref{thm:priority-period}.
Section~\ref{sec:chairman} identifies largest-deficit allocation
as a rank-one hunger game and discusses related allocation rules.
Finally, in Section~\ref{sec:monotone-score}, we introduce
monotone-score hunger games and prove
Theorems~\ref{thm:monotone-score-period}
and~\ref{thm:monotone-score-eventual}.

\section{The Hunger Game and the Common Period}
\label{sec:setup}
 In each coordinate space, $e_i$ denotes the $i$th unit vector and $\one$ the all-ones vector. For any total order $\prec$, write $i\preceq j$ when $i\prec j$ or $i=j$. Unless otherwise specified, we write $h_i$ for the $i$th coordinate of a vector $h$.
 
Throughout Sections~\ref{sec:setup}--\ref{sec:conj75}, we retain the state set $V=\{1,\ldots,n\}$, transition matrix $P$, and total order $\prec$ introduced in Section~\ref{sec:introduction}; these choices remain fixed throughout these sections. The rule of choosing the $\prec$-least state among the states of maximal hunger is the \emph{tie-breaking rule}; we call $\prec$ the \emph{tie-breaking order}.

For a hunger state $h$, let $M(h)=\max_{i\in V}h_i$ and let $\iota(h)$ be the $\prec$-least state $i$ with $h_i=M(h)$, called the \emph{firing state} of $h$. Recall that $H=P-I$, and write $H_i=e_iH$ for its $i$th row. With this notation, the hunger map introduced in Section~\ref{sec:introduction} is
\begin{equation}
 F(h)=h+H_{\iota(h)}.
 \label{eq:hunger-map}
\end{equation}

We call $(V,P,F)$ the \emph{hunger game} determined by $P$ and the fixed tie-breaking order. The hunger game is \emph{rational} if every entry of $P$ is rational. It is \emph{irreducible} if, for every $i,j\in V$, there is an integer $k\ge0$ such that $(P^k)_{ij}>0$; equivalently, every state can be reached from every other state in finitely many steps with positive probability.

The \emph{total hunger} of $h$ is $\sum_i h_i$. A firing to $i$ changes this sum by $\sum_jH_{ij}=0$, so total hunger is conserved. Moreover, $F(h+c\one)=F(h)+c\one$ for every $h\in\R^n$ and $c\in\R$. Subtracting the mean of the coordinates from the initial vector therefore translates the entire forward orbit into the subspace $Z$ defined in Section~\ref{sec:introduction}. For the rest of Sections~\ref{sec:setup}--\ref{sec:conj75}, we regard $F$ as a map from $Z$ to itself.

For $h\in Z$ and each integer $t\ge0$, the firing indices and firing vectors introduced in Section~\ref{sec:introduction} satisfy $i_t(h)=\iota(F^t(h))$ and $v_h(t)=\sum_{r=0}^{t-1}e_{i_r(h)}$, with $v_h(0)=\mathbf 0$, the all-zero vector. Thus
\begin{equation}
 F^t(h)=h+v_h(t)H \quad \text{and}\quad \sum_i(v_h(t))_i=t.
 \label{eq:orbit-formula}
\end{equation} 

From now until the end of Section~\ref{sec:conj75}, assume that $(V,P,F)$ is an irreducible rational hunger game.

Let $\pi$ be the stationary distribution of $P$, that is, the probability vector satisfying $\pi P=\pi$. By the Perron-Frobenius theorem, it is unique and strictly positive, and the left kernel of $H$ is $\R\pi$. Moreover, right multiplication by $H$ is a linear isomorphism of $Z$. Recall from Section~\ref{sec:introduction} that $T$ is the least positive integer such that $p\doteq T\pi$ is integral. Then $p$ is a positive integer vector with $pH=0$, $\sum_i p_i=T$, and $\gcd(p_1,\ldots,p_n)=1$. It follows that
\begin{equation}
 \{z\in\Z^n:zH=0\}=\Z p.
 \label{eq:integer-hunger-kernel}
\end{equation}

The following observation is implicit in \cite[Remark~3.3]{LiPropp}.

\begin{lemma}
\label{lem:period-vector}
If $h\in Z$ is recurrent with period $N$, then $v_h(N)=qp$ and $N=qT$ for some positive integer $q$.
\end{lemma}
\begin{proof}
By \eqref{eq:orbit-formula}, $v_h(N)H=0$. Equation~\eqref{eq:integer-hunger-kernel} gives $v_h(N)=qp$ for an integer $q$. Since $v_h(N)\ge0$ and $\sum_i(v_h(N))_i=N>0$, we have $q\ge1$ and $N=qT$.
\end{proof}

The following theorem shows that every recurrent orbit has least period $T$, the smallest value permitted by Lemma~\ref{lem:period-vector}, and firing vector $p$ over one period.

\begin{theorem}[Conjecture~7.4 of Li--Propp~\cite{LiPropp}]
\label{thm:conj74}
Let $(V,P,F)$ be an irreducible rational hunger game with stationary distribution $\pi$, and let $T$ be the least positive integer such that $T\pi$ is integral. Then every recurrent orbit of $F$ has least period $T$. During one period, state $i$ is fired exactly $(T\pi)_i$ times.
\end{theorem}
\subsection{The quota-return map}
\label{sec:quota-return-map}

To prove Theorem~\ref{thm:conj74}, it remains to exclude recurrent orbits of least period $qT$ with an integer $q\ge2$. We call $p_i$ the \emph{firing quota} of state $i\in V$. For $h\in Z$, define the \emph{$p$-quota time}
\begin{equation}
 \tau(h)=\min\{t\in\Z_{\ge0}:(v_h(t))_{i_t(h)}=p_{i_t(h)}\},
 \label{eq:quota-time}
\end{equation}
with $\tau(h)=\infty$ if the set is empty. Thus $\tau(h)$ is the first time $t$ at which the state $i_t(h)$ about to fire has already fired $p_{i_t(h)}$ times during the first $t$ steps.

\begin{definition}
\label{def:quota-return-map}
For every recurrent orbit $O$, define
$m_O=\min_{h\in O}M(h)$ and
$C_O=\{h\in O:M(h)=m_O\}$.
If $O$ has least period greater than $T$, then, for
$h\in C_O$ with $\tau(h)<\infty$, set
$\mathcal R_O(h)=F^{\tau(h)}(h)$.
We call $\mathcal R_O$ the quota-return map of $O$.
\end{definition}

The following lemma shows that $\mathcal R_O$ is defined at every point of $C_O$, maps $C_O$ into itself, and satisfies $\iota(h)\preceq\iota(\mathcal R_O(h))$ for every $h\in C_O$.
\begin{lemma}
\label{lem:quota-crossing}
Let $O$ be a recurrent orbit of least period greater than $T$, and let $C_O$ and $\mathcal R_O$ be as in Definition~\ref{def:quota-return-map}. For every $h\in C_O$, the following statements hold.
\begin{enumerate}[label=\textup{(\roman*)},ref=\textup{(\roman*)},leftmargin=*]
 \item\label{item:quota-time}
 $1\le\tau(h)<T$.

 \item\label{item:quota-displacement}
 Put $s=\iota(\mathcal R_O(h))$. Then $v_h(\tau(h))\le p$ coordinatewise, $v_h(\tau(h))\ne p$, and $(v_h(\tau(h)))_s=p_s$. Moreover,
 \begin{equation}
  \mathcal R_O(h)=h+\bigl(v_h(\tau(h))-p\bigr)H.
  \label{eq:quota-crossing}
 \end{equation}

 \item\label{item:quota-monotonicity}
 With $s$ as in \ref{item:quota-displacement}, we have $\mathcal R_O(h)\in C_O$, $h_s=\mathcal R_O(h)_s=m_O$, and $\iota(h)\preceq\iota(\mathcal R_O(h))$.
\end{enumerate}
In particular, $\mathcal R_O:C_O\to C_O$ is a well-defined self-map.
\end{lemma}

\begin{proof}
Fix $h\in C_O$. Suppose that $\tau(h)\ge T$. By the definition of $\tau(h)$, none of the first $T$ firings makes a coordinate of the firing vector larger than the corresponding coordinate of $p$. Hence $v_h(T)\le p$ coordinatewise. Since $\sum_i(v_h(T))_i=T=\sum_i p_i$, we have $v_h(T)=p$. Equation~\eqref{eq:orbit-formula} gives $F^T(h)=h$, contradicting the assumption that the least period of $O$ is greater than $T$. Thus $\tau=\tau(h)<T$. Since $v_h(0)=\mathbf 0$ and every $p_i$ is positive, $\tau\ge1$. This proves~\ref{item:quota-time}, so $\mathcal R_O(h)$ is defined. Put $s=\iota(F^\tau(h))$.

The definition of $\tau(h)$ gives $v_h(\tau)\le p$ and $(v_h(\tau))_s=p_s$. Since $\sum_i((v_h(\tau))_i-p_i)=\tau-T<0$, the vector $v_h(\tau)-p$ is not $\mathbf 0$. Using $pH=0$ in \eqref{eq:orbit-formula} gives $\mathcal R_O(h)=h+(v_h(\tau)-p)H$. This proves~\ref{item:quota-displacement}.

Since $(v_h(\tau)-p)_s=0$ and the off-diagonal entries of $H$ are nonnegative, we have $\mathcal R_O(h)_s=h_s+\sum_{u\ne s}((v_h(\tau))_u-p_u)P_{us}\le h_s$. Since $\mathcal R_O(h)\in O$ and $s=\iota(\mathcal R_O(h))$, it follows that
\[
 m_O\le M(\mathcal R_O(h))=\mathcal R_O(h)_s
 \le h_s\le M(h)=m_O.
\]
Every inequality in this chain is an equality, so $\mathcal R_O(h)\in C_O$ and $h_s=\mathcal R_O(h)_s=m_O$. The state $s$ also has maximal hunger at $h$, so  $\iota(h)\preceq s=\iota(\mathcal R_O(h))$. This proves~\ref{item:quota-monotonicity}.
\end{proof}

\begin{proof}[Proof of Theorem~\ref{thm:conj74}]
Let $h\in\R^n$ be recurrent, and let $N$ be its least period. We may  assume that $h\in Z$. Lemma~\ref{lem:period-vector} gives $v_h(N)=qp$ and $N=qT$ for some positive integer $q$. It remains to prove that $q=1$.

Suppose that $q\ge2$, put $O=\{h,F(h),\ldots,F^{N-1}(h)\}$, and let $\mathcal R_O$ be its quota-return map. By Lemma~\ref{lem:quota-crossing}\ref{item:quota-time} and Lemma~\ref{lem:quota-crossing}\ref{item:quota-monotonicity}, $\mathcal R_O:C_O\to C_O$ is a well-defined self-map. Since $C_O$ is finite and nonempty, $\mathcal R_O$ has a cycle $h^0,h^1,\ldots,h^{r-1}$, where $\mathcal R_O(h^j)=h^{j+1}$ and superscripts are read modulo $r$. Lemma~\ref{lem:quota-crossing}\ref{item:quota-monotonicity} gives $\iota(h^j)\preceq\iota(h^{j+1})$ for every $0\le j<r$. Thus all these firing states are equal; denote their common value by $s$.

Set $R=\sum_{j=0}^{r-1}(v_{h^j}(\tau(h^j))-p)$. By Lemma~\ref{lem:quota-crossing}\ref{item:quota-displacement}, we have $h^{j+1}=h^j+(v_{h^j}(\tau(h^j))-p)H$ for every $0\le j<r$. Summing these identities over $0\le j<r$ gives $RH=0$. Lemma~\ref{lem:quota-crossing}\ref{item:quota-displacement} also shows that each vector $v_{h^j}(\tau(h^j))-p$ is coordinatewise nonpositive and is not $\mathbf 0$. The equality $s=\iota(h^{j+1})$ gives $(v_{h^j}(\tau(h^j))-p)_s=0$ for every $0\le j<r$. Consequently, $R\le\mathbf 0$, $R\ne\mathbf 0$, and $R_s=0$. The relations $RH=0$, $R\ne\mathbf 0$, and $R_s=0$ contradict the fact that the left kernel of $H$ is spanned by the strictly positive vector $\pi$. Thus $q=1$, ending the proof.
\end{proof}

\section{The Recurrent Set as a Lattice
Fundamental Domain}
\label{sec:recurrent}
\label{sec:conj75}

Throughout this section, we retain the irreducible rational hunger game and the notation fixed in Section~\ref{sec:setup}. Let $\Lambda=Z\cap\Z^n$ and define
\begin{equation}
 L=\Lambda H
   =\operatorname{span}_{\Z}\{H_i-H_n:1\le i<n\}.
 \label{eq:hunger-lattice}
\end{equation}
Since $x\mapsto xH$ is an isomorphism on $Z$, the subgroup $L$ is a full-rank lattice in $Z$. An \emph{$L$-coset} is a set $h+L$ with $h\in Z$; we write $[h]=h+L$ for the corresponding element of $Z/L$.

\begin{proposition}
\label{prop:quota-test}
For every $h\in Z$, the conditions $h\in B$, $v_h(T)=p$, and $F^T(h)=h$ are equivalent.
\end{proposition}
\begin{proof}
Suppose first that $h\in B$. Then $h$ is recurrent, and Theorem~\ref{thm:conj74} gives $v_h(T)=p$. If $v_h(T)=p$, then \eqref{eq:orbit-formula} and $pH=0$ give $F^T(h)=h+pH=h$. Finally, $F^T(h)=h$ implies that $h$ is recurrent, so $h\in B$.
\end{proof}

Let $\mathcal S_p$ be the set of words $\sigma=(i_0,\ldots,i_{T-1})$ in which each state $i$ occurs $p_i$ times. For $\sigma\in\mathcal S_p$, let $B_\sigma$ consist of the vectors $h\in Z$ whose first $T$ firings follow $\sigma$; we call $B_\sigma$ the \emph{firing cell} of $\sigma$. Fix $\sigma=(i_0,\ldots,i_{T-1})\in\mathcal S_p$, and define $g_t=\sum_{r=0}^{t-1}H_{i_r}$ for $0\le t<T$. By \eqref{eq:orbit-formula}, the hunger vector immediately before the firing at time $t$ is $h+g_t$. Consequently, $h\in B_\sigma$ if and only if the following inequalities hold for every $t\in\{0,\ldots,T-1\}$ and every $j\in V\setminus\{i_t\}$:
\begin{equation}
 \begin{cases}
 h_{i_t}-h_j>g_{t,j}-g_{t,i_t},&\text{if }j\prec i_t,\\
 h_{i_t}-h_j\ge g_{t,j}-g_{t,i_t},&\text{if }i_t\prec j.
 \end{cases}
 \label{eq:firing-cell}
\end{equation}
The deterministic firing rule makes the sets $B_\sigma$ pairwise disjoint. Proposition~\ref{prop:quota-test} therefore gives
\begin{equation}
 B=\bigsqcup_{\sigma\in\mathcal S_p}B_\sigma.
 \label{eq:recurrent-cell-decomposition}
\end{equation}
For each $\sigma\in\mathcal S_p$, equation~\eqref{eq:firing-cell} describes $B_\sigma$ as the solution set in $Z$ of finitely many linear inequalities. In particular, $B_\sigma$ is convex. Since every state occurs in $\sigma$, these inequalities bound every coordinate difference $h_i-h_j$; together with $\sum_i h_i=0$, this shows that $B_\sigma$ is bounded. Hence \eqref{eq:recurrent-cell-decomposition} expresses $B$ as a finite disjoint union of bounded convex sets indexed by their firing words. This is the decomposition observed by Li and Propp~\cite[Remark~7.2]{LiPropp}.

Li and Propp observed that $F$ induces a translation on $Z/L$ \cite[proof of Proposition~7.6]{LiPropp}. All rows of $H$ have the same image in $Z/L$, because $H_i-H_j=(e_i-e_j)H\in L$. Therefore
\begin{equation}
 [F(h)]=[h]+[H_1]
 \label{eq:quotient-translation}
\end{equation}
for every $h\in Z$.

\begin{lemma}
\label{lem:phase-order}
The element $[H_1]\in Z/L$ has order $T$.
\end{lemma}
\begin{proof}
Since $Te_1-p\in\Lambda$ and $pH=0$, we have $TH_1\in L$. Conversely, let $a\in\Z$ satisfy $aH_1\in L$. By the definition of $L$, there is a vector $z\in\Lambda$ such that $aH_1=zH$, and hence $(ae_1-z)H=0$. By \eqref{eq:integer-hunger-kernel}, $ae_1-z=kp$ for some integer $k$. Taking the sum of the coordinates on both sides of this equality, and using $\sum_i z_i=0$ and $\sum_i p_i=T$, gives $a=kT$. Thus $a[H_1]=0$ in $Z/L$ if and only if $T$ divides $a$.
\end{proof}

For a recurrent orbit $O\subset Z$, define its set of $L$-cosets by $ \mathcal Q_L(O)=\{[x]:x\in O\}\subseteq Z/L. $

\begin{corollary}
\label{cor:orbit-cosets}
Every recurrent orbit meets exactly $T$ cosets of $L$ in $Z$. If two recurrent orbits meet a common coset of $L$, then they meet the same $T$ cosets.
\end{corollary}

\begin{proof}
Let $O$ be a recurrent orbit and fix $h\in O$. By Theorem~\ref{thm:conj74}, the orbit $O$ consists of the $T$ points $h,F(h),\ldots,F^{T-1}(h)$. Moreover, \eqref{eq:quotient-translation} gives
\[
 [F^t(h)]=[h]+t[H_1]
 \qquad \text{for $t=0,1,\ldots,T-1$.}
\]
Since $[H_1]$ has order $T$ by Lemma~\ref{lem:phase-order}, these $T$ cosets are distinct. Thus $O$ meets exactly $T$ cosets of $L$, once in each coset.

Now suppose that two recurrent orbits $O$ and $O'$ meet a common coset of $L$. Choose $x\in O$ and $x'\in O'$ with $[x]=[x']$. Applying \eqref{eq:quotient-translation} from these two points gives
\[
 \mathcal Q_L(O)
 =\{[x]+t[H_1]:0\le t<T\}
 =\{[x']+t[H_1]:0\le t<T\}
 =\mathcal Q_L(O').
\]
Hence $O$ and $O'$ meet the same $T$ cosets.
\end{proof}

We first prove the covering assertion, originally established by Li and Propp~\cite[Proposition~7.7]{LiPropp}.

\begin{proposition}
\label{prop:coverage}
Every coset of $L$ in $Z$ contains a recurrent hunger vector.
\end{proposition}

\begin{proof}
Fix $h\in Z$. Since the forward orbit of $h$ is eventually periodic, there is an integer $t_0\ge0$ such that $F^t(h)$ is recurrent for every $t\ge t_0$. Choose $k$ such that $kT\ge t_0$, and set $b=F^{kT}(h)$. Then $b$ is recurrent and belongs to the same $L$-coset as $h$.
\end{proof}

\subsection{The two-orbit quota-return map}
\label{sec:two-orbit-quota-return}

Throughout this subsection, fix two distinct recurrent orbits $O$ and $O'$, and suppose that they meet a common $L$-coset. By Corollary~\ref{cor:orbit-cosets}, the two orbits meet the same $T$ cosets, and each contains exactly one point in each of these cosets. Thus, for every $h\in O\cup O'$, the other orbit contains a unique point in the same $L$-coset as $h$. For $h\in O\cup O'$, denote this point by $\phi(h)$. Since $\phi(h)-h\in L$, there is a unique vector $z(h)\in\Lambda$ such that $\phi(h)=h+z(h)H$. Set $\lambda(h)=\max_{i\in V}\lceil z(h)_i/p_i\rceil$, and, for $0\le t\le T$, let
\[r_h(t)=z(h)-\lambda(h)p+v_{\phi(h)}(t).\]
Let $\widehat{\tau}(h)$ be the least $t\in\{0,\ldots,T-1\}$ for which
$r_h(t+1)_{i_t(\phi(h))}>0$, and set
$\widehat{\tau}(h)=\infty$ if no such time exists. Thus
$\widehat{\tau}(h)$ is the first time at which the next firing on the
orbit of $\phi(h)$ makes the selected coordinate of $r_h$ positive.
Whenever $\widehat{\tau}(h)<\infty$, define
$\mathcal R_{O,O'}(h)=F^{\widehat{\tau}(h)}(\phi(h))$. We call
$\mathcal R_{O,O'}$ the
\emph{two-orbit quota-return map} associated with $O$ and $O'$.
Since these two orbits remain fixed throughout this subsection, we
write $\mathcal R$ for $\mathcal R_{O,O'}$.

The next lemma shows that the two-orbit quota time
$\widehat{\tau}(h)$ is finite for every $h\in O\cup O'$ and records
the properties of $\mathcal R$ needed below.
\begin{lemma}
\label{lem:two-orbit-quota-properties}
For every $h\in O\cup O'$, the following statements hold.
\begin{enumerate}[
 label=\textup{(\roman*)},
 ref=\textup{(\roman*)},
 leftmargin=*
]
 \item\label{item:two-orbit-quota-time}
 $0\le\widehat{\tau}(h)<T$, so $\mathcal R(h)$ is defined.

 \item\label{item:two-orbit-quota-displacement}
 Put $s=\iota(\mathcal R(h))$. Then
 $r_h(\widehat{\tau}(h))\in
 \Z_{\le0}^n\setminus\{\mathbf 0\}$,
 $r_h(\widehat{\tau}(h))_s=0$, and
 $\mathcal R(h)=h+r_h(\widehat{\tau}(h))H$.

 \item\label{item:two-orbit-quota-bound}
 The point $\mathcal R(h)$ lies on the orbit among $O$ and $O'$
 that does not contain $h$, and
 \[
  M(\mathcal R(h))
  =h_s+\sum_{u\ne s}r_h(\widehat{\tau}(h))_uP_{us}
  \le h_s\le M(h).
 \]
\end{enumerate}
\end{lemma}

\begin{proof}
Fix $h\in O\cup O'$, and write $\lambda=\lambda(h)$. The definition of $\lambda$ gives $r_h(0)\le0$ coordinatewise. Choose $j\in V$ such that $\lceil z(h)_j/p_j\rceil=\lambda$. Since $\phi(h)$ is recurrent, Theorem~\ref{thm:conj74} gives $v_{\phi(h)}(T)=p$, and hence $r_h(T)_j=z(h)_j-(\lambda-1)p_j>0$.

Along the hunger orbit starting from $\phi(h)$, the vector $r_h(t)$ satisfies $r_h(t+1)=r_h(t)+e_{i_t(\phi(h))}$. Since $r_h(0)_j\le0<r_h(T)_j$, the $j$th coordinate becomes positive at some step $t<T$. At the first such step, the only coordinate that changes is $i_t(\phi(h))$, so $i_t(\phi(h))=j$. Hence $r_h(t+1)_{i_t(\phi(h))}>0$, and therefore $0\le\widehat{\tau}(h)<T$. This proves~\ref{item:two-orbit-quota-time}.

Write $\widehat{\tau}=\widehat{\tau}(h)$, put $y=\mathcal R(h)$, and let $s=\iota(y)$. By the minimality of $\widehat{\tau}$, we have $r_h(\widehat{\tau})\le0$. Moreover, $r_h(\widehat{\tau}+1)_s=r_h(\widehat{\tau})_s+1>0$. Since $r_h(\widehat{\tau})$ is integral, it follows that $r_h(\widehat{\tau})_s=0$.

Using $\phi(h)=h+z(h)H$ and $pH=0$, we obtain $y=\phi(h)+v_{\phi(h)}(\widehat{\tau})H =h+r_h(\widehat{\tau})H$. Moreover, $y$ lies on the same one of $O$ and $O'$ as $\phi(h)$, whereas $h$ lies on the other. Thus $y\ne h$. It follows that $r_h(\widehat{\tau})H=y-h\ne\mathbf 0$, and hence $r_h(\widehat{\tau})\ne\mathbf 0$. This proves~\ref{item:two-orbit-quota-displacement}.

Finally, by $s=\iota(y)$, we have $M(y)=y_s$. Since $r_h(\widehat{\tau})_s=0$, we have
\[
 y_s
 =h_s+\sum_{u\ne s}r_h(\widehat{\tau})_uP_{us}
 \le h_s\le M(h),
\]
where the first inequality follows from $r_h(\widehat{\tau})_u\le0$ and $P_{us}\ge0$. This proves~\ref{item:two-orbit-quota-bound}.
\end{proof}

\begin{lemma}
\label{lem:two-orbit-minima}
We have $m_O=m_{O'}$. If $m$ denotes their common value, then $\mathcal R(C_O)\subseteq C_{O'}$ and $\mathcal R(C_{O'})\subseteq C_O$. Moreover, if $h\in C_O\cup C_{O'}$ and $s=\iota(\mathcal R(h))$, then $h_s=\mathcal R(h)_s=m$ and $\iota(h)\preceq\iota(\mathcal R(h))$. Consequently, $\mathcal R$ restricts to a self-map of the finite nonempty set $C_O\cup C_{O'}$.
\end{lemma}

\begin{proof}
Choose $h\in C_O$. Since $\mathcal R(h)\in O'$,
Lemma~\ref{lem:two-orbit-quota-properties}%
\ref{item:two-orbit-quota-bound} gives $m_{O'}\le M(\mathcal R(h))\le M(h)=m_O$. Hence $m_{O'}\le m_O$. Applying the same argument to a point of $C_{O'}$ gives $m_O\le m_{O'}$. Therefore $m_O=m_{O'}=:m$.

Now let $h\in C_O\cup C_{O'}$, put $y=\mathcal R(h)$, and let $s=\iota(y)$. Since $y$ lies on the orbit among $O$ and $O'$ that does not contain $h$, and the minimum of $M$ on each orbit is $m$,
Lemma~\ref{lem:two-orbit-quota-properties}%
\ref{item:two-orbit-quota-bound} gives $m\le M(y)=y_s\le h_s\le M(h)=m$. Thus all four quantities are equal. In particular, $\mathcal R(h)$ belongs to $C_{O'}$ when $h\in C_O$, and to $C_O$ when $h\in C_{O'}$.

The equality $h_s=M(h)$ shows that $s$ is a state of maximal hunger at $h$. Since $\iota(h)$ is the first maximal state in the fixed tie-breaking order, $\iota(h)\preceq s=\iota(\mathcal R(h))$.
\end{proof}

We can now rule out the two recurrent orbits fixed at the beginning of this subsection.
\begin{lemma}
\label{lem:crossing}
The distinct recurrent orbits $O$ and $O'$ cannot meet the same coset of $L$ in $Z$.
\end{lemma}
\begin{proof}
By Lemma~\ref{lem:two-orbit-minima}, $\mathcal R$ is a self-map of the finite nonempty set $C_O\cup C_{O'}$. Choose a cycle of length $\ell$ and denote its points cyclically by $h^j$, so that $\mathcal R(h^j)=h^{j+1}$. For every $j$ modulo $\ell$, Lemma~\ref{lem:two-orbit-minima} gives $\iota(h^j)\preceq\iota(h^{j+1})$. A weakly increasing cycle in a total order is constant, so all these selected states are equal; denote their common value by $s$.

Set $R=\sum_{j=0}^{\ell-1} r_{h^j}\bigl(\widehat{\tau}(h^j)\bigr)$. By
Lemma~\ref{lem:two-orbit-quota-properties}%
\ref{item:two-orbit-quota-displacement}, for each $j$ we have $h^{j+1}=h^j+r_{h^j}\bigl(\widehat{\tau}(h^j)\bigr)H$. Summing these identities over $0\le j<\ell$ gives $RH=0$.

For every $j$, Lemma~\ref{lem:two-orbit-quota-properties}%
\ref{item:two-orbit-quota-displacement} also gives $r_{h^j}\bigl(\widehat{\tau}(h^j)\bigr)\le0$ and $r_{h^j}\bigl(\widehat{\tau}(h^j)\bigr)\ne\mathbf 0$. Moreover, the $s$th coordinate of this vector is zero because $s=\iota(h^{j+1})=\iota(\mathcal R(h^j))$. It follows that $R\le0$, $R\ne\mathbf 0$, and $R_s=0$.

The left kernel of $H$ is spanned by $\pi$. Since the chain is irreducible, every coordinate of $\pi$ is positive. Thus every nonzero vector in the left kernel of $H$ has all coordinates nonzero. This contradicts $RH=0$, $R\ne\mathbf 0$, and $R_s=0$.
\end{proof}

For $u\in L$, write $B+u=\{b+u:b\in B\}$.

\begin{theorem}[Conjecture~7.5 of Li--Propp~\cite{LiPropp}]
\label{thm:conj75}
Let $(V,P,F)$ be an irreducible rational hunger game. Then every coset of $L$ in $Z$ contains exactly one point of $B$. Equivalently,
\begin{equation}
 Z=\bigsqcup_{u\in L}(B+u).
 \label{eq:tiling}
\end{equation}
\end{theorem}

\begin{proof}
It suffices to prove that every coset of $L$ in $Z$ contains at least one and at most one point of $B$. The existence assertion is Proposition~\ref{prop:coverage}.

For uniqueness, suppose that $b,b'\in B$ belong to the same $L$-coset. If $b$ and $b'$ lie on the same recurrent orbit, then Corollary~\ref{cor:orbit-cosets} gives $b=b'$. If they lie on distinct recurrent orbits, Lemma~\ref{lem:crossing} gives a contradiction. Thus every $L$-coset contains exactly one point of $B$, which is equivalent to the disjoint-union identity \eqref{eq:tiling}.
\end{proof}


\section{Fixed-Priority Chip-Firing and Primitive Periods}
\label{sec:priority-chip-firing}

We begin with notation for the classical sequential chip-firing model. Throughout this section, let $D$ be a finite strongly connected directed multigraph with vertex set $V$, where $|V|\ge2$; loops and parallel edges are allowed. For $u,v\in V$, let $a_{uv}$ be the number of edges from $u$ to $v$, and let $d^+(u)=\sum_{v\in V}a_{uv}$ be the outdegree of $u$.

A \emph{configuration} is a vector $c\in\Z_{\ge0}^{V}$, where $c_v$ is the number of chips at vertex $v$. A vertex $v\in V$ is \emph{legal} at $c$ if $c_v\ge d^+(v)$. Firing a legal vertex $v$ replaces $c$ by the configuration $c'$ defined by $c'_u=c_u+a_{vu}-\delta_{uv}d^+(v)$ for every $u\in V$. Define the firing matrix $\Delta_D$ by $(\Delta_D)_{uv}=a_{uv}-\delta_{uv}d^+(u)$. Thus firing $v$ replaces $c$ by $c+e_v\Delta_D$. Since $\Delta_D\one^{T}=0$, every firing preserves the total number of chips.

For a finite word $\sigma=v_1\cdots v_m$ on $V$, let $|\sigma|=m$. Define its \emph{firing vector} to be the row vector $\mathbf f(\sigma)=\sum_{r=1}^{m}e_{v_{r}}\in\Z_{\ge0}^{V}$, whose $v$th coordinate is the number of times that $v$ occurs in $\sigma$, as in the firing-vector notation for the hunger game. We say that $\sigma$ is \emph{legal from $c$} if, for every $j\in\{1,\ldots,m\}$, the vertex $v_j$ is legal after $v_1,\ldots,v_{j-1}$ have fired in that order. The configuration obtained after performing $\sigma$ is $c+\mathbf f(\sigma)\Delta_D$. A \emph{legal return word from $c$} is a nonempty word $\sigma$ that is legal from $c$ and satisfies $c+\mathbf f(\sigma)\Delta_D=c$.

Since $D$ is strongly connected, it is well known that there is a unique vector $\rho\in\Z_{>0}^{V}$ satisfying $\rho\Delta_D=0$ and $\gcd_{v\in V}\rho_v=1$. Moreover, $\rho$ generates the integral left kernel of $\Delta_D$ \cite{BjornerLovaszDirected}:
\begin{equation}
 \{z\in\Z^{V}:z\Delta_D=0\}=\Z\rho.
 \label{eq:chip-integer-kernel}
\end{equation}
We call $\rho$ the \emph{primitive period vector} of $D$ and set $\ell_D=\sum_{v\in V}\rho_v$. If $\sigma$ is a legal return word, then $\mathbf f(\sigma)\Delta_D=0$. Since $\sigma$ is nonempty, \eqref{eq:chip-integer-kernel} gives $\mathbf f(\sigma)=q\rho$ for some positive integer $q$. Consequently, $|\sigma|=q\ell_D$.

These are the standard notions of sequential chip-firing. Classical results concern termination and the structure of legal firing sequences \cites{BjornerLovaszShor,BjornerLovaszDirected}; see also \cite{CorryPerkinson}. We now make the firing rule deterministic. Fix a total order $\prec$ on $V$. For a configuration $c$, let $\mathcal L(c)$ denote the set of vertices legal at $c$. The configuration is \emph{stable} if $\mathcal L(c)=\varnothing$. When $\mathcal L(c)\ne\varnothing$, let $\iota_D(c)$ be the $\prec$-least vertex in $\mathcal L(c)$. Define the \emph{fixed-priority map} $\Phi$ by
\begin{equation}
 \Phi(c)=
 \begin{cases}
  c+e_{\iota_D(c)}\Delta_D,&\text{if }\mathcal L(c)\ne\varnothing,\\
  c,&\text{if }\mathcal L(c)=\varnothing.
 \end{cases}
 \label{eq:priority-map}
\end{equation}
Thus, after every firing, the rule examines the same priority order from its beginning and fires the first vertex that is currently legal. This is a ``greedy'' rule: at each step it chooses the highest-priority vertex among those that may legally fire.

Every firing preserves nonnegativity and the total number of chips. For each integer $m\ge0$, the map $\Phi$ therefore preserves the finite set of configurations containing $m$ chips, so every forward orbit is eventually periodic. Strong connectivity and $|V|\ge2$ imply that each vertex has an outgoing edge to another vertex. Hence every legal firing changes the configuration, and the fixed points of $\Phi$ are exactly the stable configurations. A periodic orbit of $\Phi$ is called \emph{active} if it is not a stable fixed point.

In both the hunger game and the present chip-firing model, the firing vector of a periodic orbit belongs to a one-dimensional integer kernel. The kernel condition alone allows an arbitrary positive multiple of its primitive generator. The following theorem shows that the fixed-priority rule, like the hunger rule studied above, forces the primitive multiple.

\begin{theorem}
\label{thm:priority-period}
Every active periodic orbit of $\Phi$ has least period $\ell_D$. During one least period, each vertex $v$ fires exactly $\rho_v$ times.
\end{theorem}

We first record a standard removal principle for legal firing words, due to Björner and Lovász \cite[Lemma~4.3]{BjornerLovaszDirected}.

\begin{lemma}
\label{lem:priority-removal}
Let $\sigma=v_1\cdots v_m$ be legal from $c$, and suppose that $\mathbf f(\sigma)\ge\rho$ coordinatewise. For each $v\in V$, delete the first $\rho_v$ occurrences of $v$ in $\sigma$, and denote the resulting word by $\widetilde\sigma$. Then $\widetilde\sigma$ is legal from $c$ and has firing vector $\mathbf f(\widetilde\sigma)=\mathbf f(\sigma)-\rho$.
\end{lemma}

\begin{proof}
Consider an occurrence of $v$ that remains in $\widetilde\sigma$. We only need to show that $v$ is legal when this occurrence is reached after the preceding letters of $\widetilde\sigma$ have fired. For each $u\in V$, let $\gamma_u$ be the number of deleted occurrences of $u$ preceding this occurrence in $\sigma$. Since the chosen $v$ remains in $\widetilde\sigma$, we have $\gamma_v=\rho_v$ and $0\le \gamma_u\le\rho_u$ for every $u\in V$.

Let $z$ be the configuration immediately before the chosen occurrence of $v$ when $\sigma$ is performed from $c$, and let $\widetilde z$ be the configuration immediately before the same occurrence when $\widetilde\sigma$ is performed from $c$. The prefix leading to $z$ contains the firings in the prefix leading to $\widetilde z$, together with the deleted firings counted by $\gamma$. Hence $\widetilde z=z-\gamma\Delta_D$. Moreover,
\[
 (\gamma\Delta_D)_v
 =(a_{vv}-d^+(v))\rho_v+\sum_{u\ne v}\gamma_ua_{uv}
 \le(\rho\Delta_D)_v=0.
\]
Consequently, $\widetilde z_v\ge z_v$. Since $v$ is legal at $z$, it is also legal at $\widetilde z$.
\end{proof}

The central device in the proof of Theorem~\ref{thm:priority-period} is the fixed-priority analogue of the quota-return map used for the hunger game. Let $O$ be an active periodic orbit of least period $N$ whose firing vector over one period is $q\rho$, where $q\ge2$. Choose $c_0\in O$. For $0\le t<N$, set $c_t=\Phi^t(c_0)$ and $i_t=\iota_D(c_t)$, and extend both sequences to all $t\in\Z$ by $N$-periodicity. For each $t\in\Z$, delete from the indexed period word $i_ti_{t+1}\cdots i_{t+N-1}$ the first $(q-1)\rho_v$ occurrences of each vertex $v$. Applying Lemma~\ref{lem:priority-removal} successively $q-1$ times shows that the remaining word is legal from $c_t$ and has firing vector $\rho$. Let $\kappa(t)$ be the least $r\in\{t,\ldots,t+N-1\}$ for which the occurrence $i_r$ is not deleted. Note that $i_t$ itself must be deleted and the remaining word after deletion is nonempty. Then $t<\kappa(t)<t+N$. Moreover, $\kappa(t+N)=\kappa(t)+N$. Write $\Z_N\doteq\Z/N\Z$. The rule $\mathcal R_O^{\mathrm{fp}}([t])=[\kappa(t)]$ therefore defines a self-map of $\Z_N$, which we call the \emph{priority quota-return map} of $O$.

For every $t\in\Z$, we have
\begin{equation}
\begin{aligned}
 \mathbf f(i_t\cdots i_{\kappa(t)-1})
 \le(q-1)\rho,\quad
 \bigl(\mathbf f(i_t\cdots i_{\kappa(t)-1})\bigr)_{i_{\kappa(t)}}
 =(q-1)\rho_{i_{\kappa(t)}}
 \quad\text{and}\quad i_t\preceq i_{\kappa(t)}.
\end{aligned}
\label{eq:priority-quota-block}
\end{equation}
The first inequality is understood coordinatewise. The first two relations follow from the definition of $\kappa(t)$. After the first $(q-1)\rho_v$ occurrences of each vertex $v$ have been deleted from $i_ti_{t+1}\cdots i_{t+N-1}$, the resulting word is legal from $c_t$ by Lemma~\ref{lem:priority-removal}, and its first letter is $i_{\kappa(t)}$. Thus $i_{\kappa(t)}$ is legal at $c_t$. Since $i_t=\iota_D(c_t)$ is the $\prec$-least legal vertex at $c_t$, we obtain the third inequality.

\begin{proof}[Proof of Theorem~\ref{thm:priority-period}]
Let $O$ be an active periodic orbit of least period $N$. Since the configuration returns after $N$ firings, its firing vector belongs to the integral left kernel of $\Delta_D$. By \eqref{eq:chip-integer-kernel}, this firing vector is $q\rho$ for some positive integer $q$, and hence $N=q\ell_D$. It remains to prove that $q=1$.

Suppose that $q\ge2$, and form the priority quota-return map $\mathcal R_O^{\mathrm{fp}}$. Choose a cycle of this map. Starting from an integer representative $t_0$ of a point on the cycle, define $t_{j+1}=\kappa(t_j)$ until $[t_r]=[t_0]$ after $r$ steps. Since $t_{j+1}>t_j$ for every $j$, we have $t_r=t_0+wN$ for some positive integer $w$. For $0\le j<r$, the last relation in \eqref{eq:priority-quota-block} gives $i_{t_j}\preceq i_{t_{j+1}}$. Since $i_{t_r}=i_{t_0}$ by $N$-periodicity, all the vertices $i_{t_0},i_{t_1},\ldots,i_{t_r}$ are equal; denote their common value by $v$.

For $0\le j<r$, let $\beta_j$ be the firing vector of the word $i_{t_j}i_{t_j+1}\cdots i_{t_{j+1}-1}$. The intervals $[t_j,t_{j+1})$ partition $[t_0,t_0+wN)$, and therefore
\[
 \sum_{j=0}^{r-1}\beta_j=wq\rho.
\]
The first two relations in \eqref{eq:priority-quota-block} give $\beta_j\le(q-1)\rho$ and $(\beta_j)_v=(q-1)\rho_v$ for every $j$. Comparing the $v$th coordinates in the displayed equality yields $r(q-1)=wq$. Hence
\[
 \sum_{j=0}^{r-1}\bigl((q-1)\rho-\beta_j\bigr)
 =\bigl(r(q-1)-wq\bigr)\rho=\mathbf 0.
\]
Every summand is nonnegative, so $\beta_j=(q-1)\rho$ for every $j$. Consequently, $c_{t_{j+1}}=c_{t_j}+\beta_j\Delta_D=c_{t_j}$, while $t_{j+1}-t_j=(q-1)\ell_D$. This is a positive return time smaller than $N=q\ell_D$, contradicting the minimality of $N$. Therefore $q=1$, and the firing vector over one least period is $\rho$.
\end{proof}

The fixed-priority rule has also appeared as a legal toppling procedure in the study of infinite sandpiles: Fey, Levine, and Peres fix an enumeration of the sites of $\mathbb Z^d$ and, whenever unstable sites are present, topple the first unstable site in that enumeration \cite{FeyLevinePeres2010}.

Parallel chip-firing is another deterministic chip-firing dynamics that may be viewed as greedy: in each round, all legal vertices fire simultaneously \cites{BitarGoles1992,Prisner1994}. On a finite graph with a fixed total number of chips, every trajectory is eventually periodic, and its period is measured in synchronous rounds. Bitar conjectured that the period on a connected graph with $n$ vertices is at most $n$. The conjecture holds for trees \cite{BitarGoles1992}, complete graphs \cite{Levine2011}, and complete bipartite graphs \cite{Jiang2010}. However, Kiwi, Ndoundam, Tchuente, and Goles disproved the conjecture in general by constructing a family of graphs whose periods grow faster than any polynomial in the number of vertices \cite{KiwiEtAl1994}.

The periodic behavior of parallel chip-firing has developed into a substantial literature, including work on special graph classes, activity and rotation numbers, devil's-staircase phenomena, periodic firing words, and extremal periods on directed graphs \cites{DallAsta2006,Jiang2010,Levine2011,KissLevineTothmeresz2025, JiangScullyZhang2015,JiLiWang2025}.

Compared with the diverse periodic behavior of parallel chip-firing, Theorem~\ref{thm:priority-period} shows that fixed-priority chip-firing is particularly rigid: every active periodic orbit has the primitive firing vector $\rho$. A problem of particular interest is therefore to characterize, for a fixed strongly connected digraph $D$ and priority order $\prec$, the words, considered up to cyclic rotation, that occur as the firing words of active periodic orbits of $\Phi$. We may also ask how quickly a fixed-priority orbit reaches a stable configuration or enters an active periodic orbit.

\section{Greedy Deficit Rules for Chairman Assignment}
\label{sec:chairman}
The chairman assignment problem arises from the following recurring allocation problem. In its original formulation, several states form a union, state $i$ is assigned a prescribed share $\alpha_i$ of the chairmanships, and in each year the chairman is chosen from exactly one state. The task is to choose an infinite sequence of states, one for each year, such that, for every positive integer $t$, the number of chairmanships assigned to state $i$ during the first $t$ years stays as close as possible to its proportional target $t\alpha_i$. The underlying distribution problem was studied by Niederreiter, Meijer, and Tijdeman \cites{Niederreiter1972,Meijer1973,Tijdeman1973}; Tijdeman later formulated it explicitly as the chairman assignment problem \cite{Tijdeman1980}. Connections with apportionment, sequencing, and weighted chairman assignment are developed in \cites{Li2022,LiuReis2026}.

We now formalize the problem. Fix an integer $d\ge2$ and write $[d]=\{1,\ldots,d\}$. Fix a total order $\prec$ on $[d]$ and a vector $\alpha=(\alpha_1,\ldots,\alpha_d)\in(0,1)^d$ with $\sum_{i=1}^d\alpha_i=1$. An \emph{allocation word} is an infinite word $u=(u_t)_{t\ge0}$ on $[d]$, where $u_t=i$ means that state $i$ is assigned the chairmanship at time $t$. For $i\in[d]$ and $t\ge0$, let
\[
 n_i(t)=\bigl|\{s\in\{0,\ldots,t-1\}:u_s=i\}\bigr|
\]
be the number of chairmanships assigned to state $i$ during the first $t$ time steps, and write $n(t)=(n_1(t),\ldots,n_d(t))$.

After $t$ assignments, the proportional target for state $i$ is $t\alpha_i$. We therefore define its deficit at that time to be $t\alpha_i-n_i(t)$ and set $ x^{(t)}=t\alpha-n(t). $ Thus $x_i^{(t)}>0$ means that state $i$ is behind its proportional target, whereas $x_i^{(t)}<0$ means that it is ahead. Since $\sum_i n_i(t)=t$, every deficit vector belongs to $ Z_d\doteq \left\{x\in\R^d:\sum_{i=1}^d x_i=0\right\}. $ The \emph{discrepancy} of $u$ with respect to $\alpha$ is precisely the largest absolute coordinate attained by these deficit vectors:
\[
 \mathcal D_\alpha(u)
 =\sup_{t\ge1}\|x^{(t)}\|_\infty
 =\sup_{t\ge1}\max_{i\in[d]}
   |t\alpha_i-n_i(t)|.
\]

If $u_t=i$, then $n(t+1)=n(t)+e_i$, and hence
\begin{equation}
 x^{(t+1)}=x^{(t)}+\alpha-e_i.
 \label{eq:chairman-update}
\end{equation}
Thus the entire deficit vector, rather than only its largest coordinate, evolves as the state of a dynamical system. Although an allocation word naturally starts from $x^{(0)}=0$, we consider the same update rule from an arbitrary initial state in $Z_d$. A deterministic rule depending only on the current deficit is specified by a selection function $\iota:Z_d\to[d]$; its associated deficit map sends $x$ to $x+\alpha-e_{\iota(x)}$.

The most direct greedy rule selects an index with largest current deficit. Let $\iota_{\mathrm{LD}}(x)$ be the $\prec$-least index maximizing $x_i$. The associated \emph{largest-deficit map} is
\begin{equation}
 \mathcal F_\alpha(x)
 =x+\alpha-e_{\iota_{\mathrm{LD}}(x)}.
 \label{eq:largest-deficit}
\end{equation}
The map $\mathcal F_\alpha$ is precisely a rank-one hunger map. Let $P_\alpha=\one^{T}\alpha$, so that every row of $P_\alpha$ is equal to $\alpha$. Then $P_\alpha$ is a positive stochastic matrix with stationary distribution $\alpha$, and row $i$ of $P_\alpha-I$ is $\alpha-e_i$. Its hunger map selects the $\prec$-least index maximizing $x_i$ and then adds $\alpha-e_i$. It therefore agrees with $\mathcal F_\alpha$ on $Z_d$. Thus largest-deficit chairman assignment is exactly the hunger game for the rank-one chain $P_\alpha$.

Consequently, if $\alpha$ is rational and $Q$ is the least positive integer such that $Q\alpha$ is integral, then every forward orbit of $\mathcal F_\alpha$ eventually enters a cycle of least period $Q$. During one period, state $i$ is chosen exactly $Q\alpha_i$ times.

Several related works study a similar greedy allocation rule using the \emph{add-then-select} convention defined below \cites{AdlerEtAl2003,CoppersmithEtAl2011,HarrensteinEtAl2022}. Let $\widehat\iota_\alpha(x)$ be the $\prec$-least index maximizing $x_i+\alpha_i$. The corresponding map is
\begin{equation}
 \widehat{\mathcal F}_\alpha(x)
 =x+\alpha-e_{\widehat\iota_\alpha(x)}.
 \label{eq:add-then-select}
\end{equation}
This rule appears as a proportionate-fair scheduler in \cite{AdlerEtAl2003}. If $\alpha$ is rational and $Q$ is as above, \cite[Lemma~4]{AdlerEtAl2003} implies that the add-then-select orbit starting from zero returns to zero after $Q$ steps. This is its least period, since a return after $N$ steps requires $N\alpha$ to be integral. Coppersmith, Nowicki, Paleologo, Tresser, and Wu study an online extension in which the target proportions may change from one step to the next \cite{CoppersmithEtAl2011}. Measuring discrepancy from the cumulative target shares, they prove that the add-then-select rule has sharp worst-case bound $\sum_{k=2}^d \frac{1}{k}$. Moreover, no rule that makes each selection without knowing the future target proportions can guarantee a smaller bound \cite[Theorem~3.1]{CoppersmithEtAl2011}. For a fixed target vector $\alpha$, the modified Frege method of Harrenstein, Lackner, and Lackner is precisely the add-then-select map $\widehat{\mathcal F}_\alpha$ defined in \eqref{eq:add-then-select}. They study its quota properties and long-term proportionality \cite[Theorems~2 and~3]{HarrensteinEtAl2022}.

Meijer proved that every target vector $\alpha$ admits an allocation
word $u$ with $\mathcal{D}_\alpha(u)\le 1-\frac{1}{2d-2}$, and that no smaller
bound works for all $\alpha$ \cite{Meijer1973}.
Tijdeman introduced an online allocation rule achieving this bound
\cite[Theorem~1]{Tijdeman1980}. We use the two-parameter formulation of Berth\'e et al.\ \cite{BertheEtAl2024}. Fix parameters $C\ge0$ and $C'\ge1-\frac1d$. For $x\in Z_d$, define the set of eligible states by
\[
 E_{C'}(x)
 =\{i\in[d]:x_i+\alpha_i-1\ge-C'\}.
\]
The condition on $C'$ ensures that $E_{C'}(x)$ is nonempty for every $x\in Z_d$. Moreover, if state $i$ is eligible, then selecting it does not move its deficit below $-C'$. For an eligible state $i$, the quantity $(C-x_i)/\alpha_i$ measures how long its deficit would take to reach the upper level $C$ if it continued to increase at rate $\alpha_i$. Tijdeman's rule selects the eligible state with the smallest such waiting time. Equivalently, with values in $\overline{\R}=\R\cup\{-\infty,+\infty\}$, define the thresholded score
\begin{equation}
 \psi_i^{C,C'}(r)=
 \begin{cases}
  \dfrac{r-C}{\alpha_i},
     & r+\alpha_i-1\ge-C',\\[5pt]
  -\infty,
     & r+\alpha_i-1<-C'.
 \end{cases}
 \label{eq:tijdeman-score}
\end{equation}
Let $\iota_{\mathrm{Tij}}(x)$ be the $\prec$-least index maximizing $\psi_i^{C,C'}(x_i)$. Tijdeman's rule selects $\iota_{\mathrm{Tij}}(x)$ and updates the deficit according to \eqref{eq:chairman-update}.

The largest-deficit rule, the add-then-select rule, and Tijdeman's rule all evolve the deficit vector according to \eqref{eq:chairman-update}. They differ in how the index $i$ is selected. At a state $x$, the largest-deficit rule selects the $\prec$-least index maximizing $x_i$, whereas the add-then-select rule maximizes $x_i+\alpha_i$. Tijdeman's rule selects the $\prec$-least index maximizing $\psi_i^{C,C'}(x_i)$. Each of these dynamics combines the common update \eqref{eq:chairman-update} with a coordinatewise score used to select the next state. This common structure motivates the generalized hunger dynamics introduced in the next section.
\section{Monotone-Score Hunger Games}
\label{sec:monotone-score}

We return to the hunger-game notation of Section~\ref{sec:setup}. Thus $V=\{1,\ldots,n\}$, $P$ is an irreducible rational transition matrix, $\prec$ is the fixed tie-breaking order, and $H=P-I$. We retain the stationary distribution $\pi$, the integer $T$, and the primitive vector $p=T\pi$. In particular, $pH=0$, $\sum_i p_i=T$, and \eqref{eq:integer-hunger-kernel} describes the integral left kernel of $H$.

Recall from \eqref{eq:hunger-map} that, at a hunger vector $h=(h_1,\ldots,h_n)\in\R^n$, the ordinary hunger map selects $\iota(h)$, the $\prec$-least state $i$ maximizing $h_i$, and then adds $H_i$. We now replace the coordinate $h_i$ in this selection rule by a state-dependent score, while leaving the increments $H_i$ unchanged.

Recall that $\overline{\R}=\R\cup\{-\infty,+\infty\}$, and for each $i\in V$ fix a nondecreasing function $\varphi_i:\R\to\overline{\R}$ as its \emph{score function}. For $h\in\R^n$, let $\iota_\varphi(h)$ be the $\prec$-least state maximizing $\varphi_i(h_i)$. Since $V$ is finite, this state is always defined. The associated \emph{monotone-score hunger map} is
\begin{equation}
 F_\varphi(h)=h+H_{\iota_\varphi(h)}.
 \label{eq:monotone-score-map}
\end{equation}
Thus \eqref{eq:monotone-score-map} differs from \eqref{eq:hunger-map} only in the selector. Taking $\varphi_i(u)=u$ for every $i$ recovers the ordinary hunger map $F$. As before, $F_\varphi$ preserves total hunger because $H\one^{T}=0$. Unlike $F$, it need not satisfy $F_\varphi(h+c\one)=F_\varphi(h)+c\one$. We therefore regard $F_\varphi$ as a map on $\R^n$, rather than restricting without loss of generality to $Z$.

We use the notation introduced before \eqref{eq:orbit-formula}, now with respect to $F_\varphi$. For $h\in\R^n$ and $t\ge0$, let $i_t(h)=\iota_\varphi(F_\varphi^t(h))$; thus $i_t(h)$ is the state selected by the monotone-score rule at time $t$. Define $v_h(t)=\sum_{r=0}^{t-1}e_{i_r(h)}$, with $v_h(0)=0$, so that $v_h(t)$ records the first $t$ firings of $F_\varphi$. The corresponding orbit identity is $F_\varphi^t(h)=h+v_h(t)H$. We also use the recurrence terminology from Sections~\ref{sec:introduction} and \ref{sec:setup}, with $F$ replaced by $F_\varphi$. A state $h\in\R^n$ is \emph{recurrent for $F_\varphi$} if $F_\varphi^N(h)=h$ for some positive integer $N$. Its \emph{least period} is the least such $N$, and the corresponding periodic orbit is $\{h,F_\varphi(h),\ldots,F_\varphi^{N-1}(h)\}$. Its firing vector over one period is $v_h(N)$.

\begin{theorem}
\label{thm:monotone-score-period}
Every recurrent orbit of $F_\varphi$ has least period $T$. During one least period, state $i$ is fired exactly $p_i$ times.
\end{theorem}

The proof of Lemma~\ref{lem:period-vector} uses only the orbit identity and \eqref{eq:integer-hunger-kernel}; neither step depends on the rule used to select the states $i_t(h)$. Consequently, if a recurrent orbit of $F_\varphi$ has least period $N$, then its firing vector is $qp$ and $N=qT$ for some positive integer $q$. The remaining argument for proving Theorem~\ref{thm:monotone-score-period} is the monotone-score counterpart of Lemma~\ref{lem:quota-crossing} and Theorem~\ref{thm:conj74}. We shall include the full argument for completeness.

For the next lemma, fix a periodic orbit $O$ of $F_\varphi$ whose least period $N$ is greater than $T$. For $h\in O$, put $M_\varphi(h)\doteq\max_{i\in V}\varphi_i(h_i)$, and define $m_O^\varphi\doteq\min_{h\in O}M_\varphi(h)$ and $ C_O^\varphi \doteq\{h\in O:M_\varphi(h)=m_O^\varphi\}. $ For $h\in C_O^\varphi$, define $\tau(h)$ by \eqref{eq:quota-time}, using the present firing vector $v_h(t)$. Whenever $\tau(h)<\infty$, set $\mathcal R_O^\varphi(h)=F_\varphi^{\tau(h)}(h)$. Thus $\mathcal R_O^\varphi$ is the monotone-score analogue of the quota-return map in Definition~\ref{def:quota-return-map}.
\begin{lemma}
\label{lem:score-quota-crossing}
For every $h\in C_O^\varphi$, the following statements hold.
\begin{enumerate}[
label=\textup{(\roman*)}, ref=\textup{(\roman*)}, leftmargin=* ]
\item\label{item:score-quota-time} $1\le\tau(h)<T$, so $\mathcal R_O^\varphi(h)$ is defined.

\item\label{item:score-quota-displacement} Put $h'=\mathcal R_O^\varphi(h)$ and $s=\iota_\varphi(h')$. Then $v_h(\tau(h))-p\in \Z_{\le0}^n\setminus\{\mathbf 0\}$, $\bigl(v_h(\tau(h))-p\bigr)_s=0$, and
 \[
  h'=h+\bigl(v_h(\tau(h))-p\bigr)H.
 \]

\item\label{item:score-quota-monotonicity} With $h'$ and $s$ as in \ref{item:score-quota-displacement}, we have $h'\in C_O^\varphi$, $\varphi_s(h'_s)=\varphi_s(h_s)=m_O^\varphi$, and $\iota_\varphi(h)\preceq\iota_\varphi(h')$.
\end{enumerate}
Consequently, $\mathcal R_O^\varphi:C_O^\varphi\to C_O^\varphi$ is a well-defined self-map.
\end{lemma}

\begin{proof}
Let $N$ be the least period of $O$. The argument immediately following Theorem~\ref{thm:monotone-score-period} gives $N=qT$ and firing vector $qp$ for some integer $q\ge2$. Thus, during one period starting from any $h\in O$, state $i$ is fired $qp_i>p_i$ times. The defining set for $\tau(h)$ is therefore nonempty.

Suppose that $\tau(h)\ge T$. By the definition of $\tau$, each state $i$ is fired at most $p_i$ times during the first $T$ steps. Hence $v_h(T)\le p$ coordinatewise. Since $\sum_i(v_h(T))_i=T=\sum_i p_i$, we must have $v_h(T)=p$. It follows that $F_\varphi^T(h)=h+pH=h$, contradicting $N>T$. Moreover, $v_h(0)=0$ and $p_{i_0(h)}>0$, so $\tau(h)\ne0$. Therefore $1\le\tau(h)<T$, proving~\ref{item:score-quota-time}.

Write $\tau=\tau(h)$. By definition, $\tau$ is the first time $t$ at which the state $i_t(h)$, selected for the next firing, has already been fired $p_{i_t(h)}$ times. Consequently, $v_h(\tau)\le p$ coordinatewise and, for $s=i_\tau(h)=\iota_\varphi(h')$, we have $(v_h(\tau))_s=p_s$. Since $\sum_i(v_h(\tau))_i=\tau<T=\sum_i p_i$, the vector $v_h(\tau)-p$ is not $\mathbf 0$. Using $pH=0$, we obtain
\[
 h'=h+v_h(\tau)H
   =h+\bigl(v_h(\tau)-p\bigr)H.
\]
This proves~\ref{item:score-quota-displacement}.

Since $\bigl(v_h(\tau)-p\bigr)_s=0$ and the off-diagonal entries of $H$ are nonnegative,
\[
 h'_s
 =h_s+\sum_{u\ne s}
   \bigl((v_h(\tau))_u-p_u\bigr)P_{us}
 \le h_s.
\]
From $h'\in O$ and $s=\iota_\varphi(h')$, we obtain $M_\varphi(h')=\varphi_s(h'_s)$. Since $\varphi_s$ is nondecreasing,
\[
 m_O^\varphi
 \le M_\varphi(h')
 =\varphi_s(h'_s)
 \le\varphi_s(h_s)
 \le M_\varphi(h)
 =m_O^\varphi.
\]
Since the first and last quantities are equal, all inequalities in this chain are equalities. In particular, $M_\varphi(h')=m_O^\varphi$, so $h'\in C_O^\varphi$, and $\varphi_s(h_s)=M_\varphi(h)$, so $s$ also maximizes the score at $h$. Since $\iota_\varphi(h)$ is the $\prec$-least score maximizer at $h$, we have $\iota_\varphi(h)\preceq s=\iota_\varphi(h')$. This proves~\ref{item:score-quota-monotonicity}.
\end{proof}

\begin{proof}[Proof of Theorem~\ref{thm:monotone-score-period}]
Let $O$ be a recurrent orbit of least period $N$. We noted in the paragraph immediately following the statement of Theorem~\ref{thm:monotone-score-period} that $N=qT$ and that the firing vector of $O$ over one period is $qp$, for some positive integer $q$. It remains to prove that $q=1$.

Suppose that $q\ge2$. By Lemma~\ref{lem:score-quota-crossing}\ref{item:score-quota-time} and
Lemma~\ref{lem:score-quota-crossing}%
\ref{item:score-quota-monotonicity}, $\mathcal R_O^\varphi$ is a self-map of the finite nonempty set $C_O^\varphi$. Choose a cycle $h^0,h^1,\ldots,h^{r-1}$ of this map, with $\mathcal R_O^\varphi(h^j)=h^{j+1}$, where the superscripts are read
modulo $r$. Lemma~\ref{lem:score-quota-crossing}%
\ref{item:score-quota-monotonicity} gives $\iota_\varphi(h^j)\preceq\iota_\varphi(h^{j+1})$ for every $j$. These inequalities form a weakly increasing cycle in the total order $\prec$, so all the states $\iota_\varphi(h^j)$ are equal; denote their common value by $s$.

Set $ R=\sum_{j=0}^{r-1} \bigl(v_{h^j}(\tau(h^j))-p\bigr). $
By Lemma~\ref{lem:score-quota-crossing}%
\ref{item:score-quota-displacement}, for every $j$,
\[
 h^{j+1}
 =h^j+\bigl(v_{h^j}(\tau(h^j))-p\bigr)H.
\]
Summing these identities over $0\le j<r$ gives $RH=0$.
Lemma~\ref{lem:score-quota-crossing}%
\ref{item:score-quota-displacement} also shows that, for every $0\le j<r$, the vector $v_{h^j}(\tau(h^j))-p$ is coordinatewise nonpositive and is not $\mathbf 0$. Moreover, $s=\iota_\varphi(h^{j+1})$, so $\bigl(v_{h^j}(\tau(h^j))-p\bigr)_s=0$. Consequently, $R\le0$, $R\ne\mathbf 0$, and $R_s=0$. Since the left kernel of $H$ is spanned by the strictly positive vector $\pi$, the conditions $RH=0$ and $R\ne\mathbf 0$ force $R$ to be a nonzero scalar multiple of $\pi$, contradicting $R_s=0$. This completes the proof.
\end{proof}

Theorem~\ref{thm:monotone-score-period} determines the periods of recurrent states but does not assert that every forward orbit reaches one. To obtain eventual periodicity for every initial state, we impose the following \emph{pairwise non-domination condition}:
\begin{equation}
 \inf_{u\in\R}\varphi_i(u)
 <
 \sup_{u\in\R}\varphi_j(u)
 \qquad\text{for every distinct }i,j\in V.
 \label{eq:score-no-domination}
\end{equation}

\begin{theorem}
\label{thm:monotone-score-eventual}
If the score functions satisfy \eqref{eq:score-no-domination}, then every forward orbit of $F_\varphi$ is bounded and eventually periodic. Its eventual periodic orbit has least period $T$ and firing vector $p$.
\end{theorem}
\begin{proof}
If $n=1$, then $P=(1)$ and $H=0$, so every point is fixed. Assume henceforth that $n\ge2$. Fix $h^{(0)}\in\R^n$, write $h^{(t)}=F_\varphi^t(h^{(0)})$, and set $S=\sum_{i\in V}h_i^{(0)}$. Since every row of $H$ has coordinate sum zero, $\sum_{i\in V}h_i^{(t)}=S$ for every $t\ge0$. By \eqref{eq:score-no-domination}, for each ordered pair of distinct states $i,j$, we may choose $A_{ij},B_{ij}\in\R$ such that $\varphi_i(A_{ij})<\varphi_j(B_{ij})$. Define
\[
 K_i=\min\left\{
   \min_{j\ne i}A_{ij},
   \ S-\sum_{j\ne i}B_{ij}
 \right\}.
\]
Fix $t\ge0$ and put $s=\iota_\varphi(h^{(t)})$. We claim that $h_s^{(t)}>K_s$. If $h_s^{(t)}>\min_{j\ne s}A_{sj}$, this follows from the definition of $K_s$. Otherwise, $h_s^{(t)}\le A_{sj}$ for every $j\ne s$. If $h_j^{(t)}\ge B_{sj}$ for some $j\ne s$, then $\varphi_s(h_s^{(t)})\le\varphi_s(A_{sj}) <\varphi_j(B_{sj})\le\varphi_j(h_j^{(t)})$. This is impossible because $s=\iota_\varphi(h^{(t)})$ maximizes the scores at $h^{(t)}$. Hence $h_j^{(t)}<B_{sj}$ for every $j\ne s$. Since the coordinates of $h^{(t)}$ sum to $S$, we obtain $h_s^{(t)}=S-\sum_{j\ne s}h_j^{(t)} >S-\sum_{j\ne s}B_{sj}\ge K_s$, proving the claim.

Set
\[
 b=\min\left\{
   \min_{i\in V}h_i^{(0)},
   \ \min_{i\in V}(K_i-1)
 \right\}.
\]
The coordinate $h_i^{(t)}$ can decrease only when state $i$ is fired. If $i$ is fired at time $t$, then the claim gives $h_i^{(t)}>K_i$, and hence $h_i^{(t+1)}=h_i^{(t)}+P_{ii}-1>K_i-1\ge b$. Firing any other state changes the $i$th coordinate by a nonnegative amount. Since $h_i^{(0)}\ge b$, it follows that $h_i^{(t)}\ge b$ for every $i\in V$ and $t\ge0$. The identity $\sum_i h_i^{(t)}=S$ then gives $h_i^{(t)}\le S-(n-1)b$. Thus every point $h^{(t)}$ of the forward orbit lies in the bounded box $[b,S-(n-1)b]^n$.

Because $P$ is rational, there is a positive integer $r_0$ such that $r_0H$ has integer entries. Since each update adds a row of $H$, every $h^{(t)}$ belongs to $h^{(0)}+r_0^{-1}\Z^n$. The intersection of this discrete set with $[b,S-(n-1)b]^n$ is finite. Hence the forward orbit of $h^{(0)}$ is eventually periodic. Theorem~\ref{thm:monotone-score-period} shows that its eventual periodic orbit has least period $T$ and firing vector $p$.
\end{proof}
\begin{remark}
     Similar arguments show that Theorems~\ref{thm:conj74} and \ref{thm:monotone-score-eventual} remain valid when the rationality assumption on $P$ is replaced by rationality of its stationary distribution $\pi$, with all other assumptions unchanged.
\end{remark}

\begin{example}
\label{ex:score-strictness}
The strict inequality in the hypothesis \eqref{eq:score-no-domination} of Theorem~\ref{thm:monotone-score-eventual} cannot be replaced by a non-strict inequality. Let $V=\{1,2\}$, let $P_{ij}=1/2$ for all $i,j\in V$, and fix the order $1\prec2$. Take the constant score functions $\varphi_1=\varphi_2\equiv0$.
Starting from $h^{(0)}=(0,0)$, the two scores are always equal, so state $1$ fires at every step. Since firing state $1$ adds $(-1/2,1/2)$, we have $h^{(t)}=(-t/2,t/2)$ for every $t\ge0$. This forward orbit is unbounded and is not eventually periodic.
\end{example}

\subsection*{Acknowledgements and the use of AI}
We thank Professor Yaokun Wu for introducing us to this topic in his course in 2024; see a lecture note \url{https://math.sjtu.edu.cn/faculty/ykwu/data/TeachingMaterial/hunger_game.pdf}.  Our investigation of the period conjecture (Theorem~\ref{thm:conj74}) of Li and Propp led us to consider
the quota-return map, which stops just before the first firing
that would exceed a state's quota. In subsequent discussions,
OpenAI's GPT-5.6 Sol observed that this map sends $C_O$ into
itself and completed the proof of Theorem~\ref{thm:conj74}.
We then worked with the model to develop further results
using the same idea. We also used several large language models to help revise
and polish the exposition. The authors independently verified every mathematical argument and take full responsibility for all
claims and proofs in this paper.

\providecommand{\bysame}{\leavevmode\hbox to3em{\hrulefill}\thinspace} \providecommand{\MR}{\relax\ifhmode\unskip\space\fi MR }

\end{document}